\documentclass{amsart}

\usepackage{amsmath}
\usepackage{amssymb}
\usepackage{amsthm}
\usepackage{esint}
\usepackage[inline]{enumitem}
\usepackage{url}
\usepackage{color}

\theoremstyle{plain}
\newtheorem{theorem}{Theorem}[section]
\newtheorem{corollary}[theorem]{Corollary}
\newtheorem{lemma}[theorem]{Lemma}
\newtheorem{proposition}[theorem]{Proposition}

\theoremstyle{definition}
\newtheorem{definition}[theorem]{Definition}
\newtheorem{remark}[theorem]{Remark}

\theoremstyle{remark}
\newtheorem{example}[theorem]{Example}

\numberwithin{theorem}{section}
\numberwithin{equation}{section}

\newcommand{\R}{\mathbb{R}}

\newcommand{\dist}{\mathrm{dist}}

\newcommand{\cl}{\overline}
\newcommand{\del}{\partial}
\newcommand{\loc}{\mathrm{loc}}

\DeclareMathOperator{\divergence}{div}

\DeclareMathOperator*{\essinf}{ess\,inf}

\DeclareMathOperator*{\spt}{supp}
\newcommand{\capacity}{\mathrm{cap}}

\newcommand{\wkto}{\rightharpoonup}

\newcommand{\trinorm}[1]
{{
    \left\vert\kern-0.20ex\left\vert\kern-0.20ex\left\vert
    #1 
    \right\vert\kern-0.20ex\right\vert\kern-0.20ex\right\vert
}}

\newcommand{\M}{\mathcal{M}}

\renewcommand{\epsilon}{\varepsilon}
\renewcommand{\theta}{\vartheta}
\renewcommand{\kappa}{\varkappa}
\renewcommand{\rho}{\varrho}
\renewcommand{\phi}{\varphi}

\begin{document}


\title[Elliptic equations with singular nonlinearity]
{A stability result for elliptic equations with singular nonlinearity and its applications to homogenization problems}
\author{Takanobu Hara}
\email{takanobu.hara.math@gmail.com}
\address{Department of Mathematics, Hokkaido University,
                Kita 8 Nishi 10  Sapporo, Hokkaido 060-0810, Japan}
\date{\today}


\begin{abstract}
We consider model semilinear elliptic equations of the type
\[
\begin{cases}
- \mathrm{div} (A(x) \nabla u) =  f u^{- \lambda}, \quad u > 0 \quad \text{in} \ \Omega,
\\
u \in H_{0}^{1}(\Omega),
\end{cases}
\]
where $\Omega$ is a bounded domain in $\mathbf{R}^{N}$, $N \ge 1$,
$A \in L^{\infty}(\Omega)^{N \times N}$ is a coercive matrix,
$0 < \lambda \le 1$ and $f$ is a nonnegative function in $L^{1}_{loc}(\Omega)$,
or more generally, nonnegative Radon measure on $\Omega$.
We discuss $H^{1}$-stability of $u$ under a minimal assumption on $f$.
Additionally, we apply the result to homogenization problems.
\end{abstract}


\maketitle


\section{Introduction and main results}\label{sec:introduction}

Let $\Omega$ be a bounded domain in $\R^{N}$ ($N \ge 1$).
We consider model elliptic differential equations of the type
\begin{equation}\label{eqn:DP}
\begin{cases}
\displaystyle
- \divergence (A(x) \nabla u) = \frac{\sigma}{u^{\lambda}}, \quad u > 0 \quad \text{in} \ \Omega,
\\
u \in H_{0}^{1}(\Omega),
\end{cases}
\end{equation}
where $H_{0}^{1}(\Omega)$ is the closure of $C_{c}^{\infty}(\Omega)$ 
with respect to the norm $\| \nabla \cdot \|_{L^{2}(\Omega)}$,
$A$ is a coercive matrix-valued function in $M(\alpha, \beta, \Omega)$
(see Sect. \ref{sec:SEP} for details),
$\lambda$ is a fixed positive constant,
and $\sigma$ is a nonnegative function in $L^{1}_{\loc}(\Omega)$,
or more generally, nonnegative (locally finite) Radon measure on $\Omega$.

Elliptic boundary problems of the type \eqref{eqn:DP} are said to be singular
from the difficulty to control near $u = 0$.
The physical background of the problems lies
in heat conduction in conductive materials,
binary communication by signals,
or the theory of pseudoplastic fluids.
Mathematical study involved has been ongoing, especially since the work by
Crandall, Rabinowitz and Tartar \cite{MR427826}.
For basics of the singular elliptic problems, we refer to
\cite{MR548961, MR860919, MR1458503, MR1784140, MR2064102, MR2359771, MR3712944}
and the references therein.

Traditionally, $\sigma$ was assumed to be a continuous or locally bounded function.
In \cite{MR2592976}, Boccardo and Orsina studied $L^{p}$-integrable $\sigma$.
After their work, singular elliptic problems with measure valued coefficients
have been studied extensively (e.g., \cite{MR3579130,MR3712944, MR3829750}).
Boccardo and Casado-D\'{\i}az \cite{MR3177473} applied the results
to H (or G)-convergence of elliptic operators.
Extension to other types homogenization problems and further developments
can be found in \cite{MR3583569, MR3583469, MR3802572}.

It is a delicate problem whether we may restrict the space for finding
solutions to \eqref{eqn:DP} to just $H_{0}^{1}(\Omega)$.
If $u$ vanishes rapidly near the boundary, then its Dirichlet integral diverges.
To avoid this, $\lambda$ and $\sigma$ must satisfy a certain relation.
However, if we change the class to find solutions, for example, $H_{\loc}^{1}(\Omega) \cap C(\cl{\Omega})$,
the relation between them changes naturally
(see, e.g., \cite{MR1866062, hara2021trace} and the above references).
Note that there is no inclusion between $H_{0}^{1}(\Omega)$ and
the class of classical solutions $C^{2}(\Omega) \cap C(\cl{\Omega})$.

Below, we only consider (finite energy) weak solutions in $H_{0}^{1}(\Omega)$.
For simplicity, we assume that $A(x) = I$ and $N \ge 2$ (see \cite{MR548961} for $N =1$).
The first significant result was given by Lazer and McKenna \cite{MR1037213}.
They proved that if $0 < \sigma \in C^{\alpha}(\cl{\Omega})$ and $\partial \Omega$ is sufficiently smooth,
then there exists a weak solution to \eqref{eqn:DP} if and only if $\lambda < 3$.
Zhang and Cheng \cite{MR2064102} extended the result to $\sigma(x) := \dist(x, \partial \Omega)^{-s}$.
More generally, D\'{\i}az, Hern\'{a}ndez and Rakotoson \cite{MR2831448} considered
\begin{equation}\label{eqn:DHR_sigma}
\sigma(x) = f(x) \, \dist(x, \partial \Omega)^{-s}, \quad 0 < c \le f(x) \in L^{\infty}(\Omega).
\end{equation}
They proved that Eq. \eqref{eqn:DP} has a weak solution if and only if
\begin{equation}\label{eqn:DHR}
s < (3 - \lambda) / 2.
\end{equation}
Meanwhile, Boccardo and Orsina \cite{MR2592976}
proved that if $0 \le \sigma \in L^{m}(\Omega)$, $0 < \lambda \le 1$ and
\begin{equation}\label{eqn:cond_BO}
m \ge \left( \frac{2^{*}}{1 - \lambda} \right)',
\end{equation}
then there exists a weak solution to \eqref{eqn:DP},
where $2^*$ is the Sobolev conjugate of $2$ (see, \eqref{eqn:sobolev}).
When $\lambda > 1$, the corresponding solution may not have finite energy in general.
In fact, Sun and Zhang \cite{MR3168615} proved that if $0 < \sigma \in L^{1}(\Omega)$ and 
$\lambda > 1$, then there exists a weak solution to \eqref{eqn:DP} if and only if
\[
\exists u_{0} \in H_{0}^{1}(\Omega) \ \text{s.t.} \ 
\int_{\Omega} |u_{0}|^{1 - \lambda} \sigma(x) \, dx < \infty.
\]
They also prove that there is no such a function if $\sigma \equiv 1$ and $\lambda \ge 3$. 
Further information can be found in the paper by Oliva and Petitta \cite{MR3712944}.

Let us assume that $0 < \lambda < 1$ and $A$ is symmetric.
Then, we can consider the functional 
\[
J(u)
=
\frac{1}{2} \int_{\Omega} A \nabla u \cdot \nabla u \, dx
-
\frac{1}{1 - \lambda} \int_{\Omega} u_{+}^{1 - \lambda} \, d \sigma.
\]
Since $J$ is convex on $\mathcal{K} := \{ u \in H_{0}^{1}(\Omega) \colon u \ge 0 \text{ a.e. in } \Omega \}$,
if it is bounded from below, we can get a unique minimizer $u$ by using the direct method.
The Euler-Lagrange equation of $J$ can not be derived directly because it is not differentiable,
but we can show that $u$ is a weak solution to \eqref{eqn:DP}
by making appropriate approximations (for the precise meaning, see, e.g., \cite[Remark 5.4]{MR2592976}).
The remaining problem, the boundedness of $J$,
is equivalent to the validity of the $L^{2}(dx)$-$L^{1 - \lambda}(d \sigma)$ trace inequality
\begin{equation}\label{eqn:trace_lambda}
\| \varphi \|_{L^{1 - \lambda}(\Omega; \sigma)}
\le C \| \nabla \varphi \|_{L^{2}(\Omega)}, \quad \forall \varphi \in C_{c}^{\infty}(\Omega).
\end{equation}
As mentioned in later, \eqref{eqn:DHR_sigma}-\eqref{eqn:DHR} or \eqref{eqn:cond_BO}
are sufficient for \eqref{eqn:trace_lambda}, but not necessary.
In \cite{hara2021trace}, the author prove that
there is a unique weak solution $u$ to \eqref{eqn:DP}
if and only if \eqref{eqn:trace_lambda} holds.
Note that if \eqref{eqn:trace_lambda} holds, then $\sigma$
must be absolutely continuous with respect to the variational capacity.

In this paper, we define a distance of measures satisfying \eqref{eqn:trace_lambda}
and discuss continuous dependency of $u$ (Theorem \ref{thm:stability}).
Additionally, we apply the result to homogenization problems (Theorem \ref{thm:main2}).
In conclusion, we observe that weak solutions to \eqref{eqn:DP} 
are stable for perturbations of $A$ and $\sigma$.
These results can be regarded as refinements of \cite[Theorems 2.7 and 3.4]{MR3177473}.
The novelty of our results is that the class of $\sigma$ has been extended
to be necessary and sufficient for the existence of $H_{0}^{1}$-weak solutions.
We investigate various properties of measures satisfying \eqref{eqn:trace_lambda}
and compare them with concrete examples.

Basically, we do not assume the smoothness of $\partial \Omega$.
It will be assumed when we observe specific examples of $\sigma$
satisfying \eqref{eqn:trace_lambda}.
For example, bounded Lipschitz domains satisfy the conditions in Example \ref{example:ext+int}.

\subsection*{Organization of the paper}
In Sect. \ref{sec:PT}, we present auxiliary results from potential theory.
In Sect. \ref{sec:TR}, we investigate some properties of
$\sigma$ satisfying \eqref{eqn:trace_lambda}.
In Sect. \ref{sec:SEP}, we prove a stability result for weak solutions to the singular elliptic problems.
In Sect. \ref{sec:HG}, we apply the results in Sections \ref{sec:TR} and \ref{sec:SEP} to homogenization problems.

\subsection*{Acknowledgments}
This work was supported by JSPS KAKENHI Grant Number JP17H01092.

\subsection*{Notation}
We use the following notation.
Throughout the paper, $\Omega$ is a domain (connected open subset) in $\R^{N}$.
\begin{itemize}
\item
$\mathbf{1}_{E}(x) :=$ the indicator function of a set $E$.
\item
$C_{c}^{\infty}(\Omega) :=$
the set of all infinitely-differentiable functions with compact support in $\Omega$.
\item
$\M^{+}(\Omega) :=$ the set of all nonnegative Radon measures on $\Omega$.
\item
$L^{p}(\Omega; \mu) :=$ the $L^{p}$ space with respect to $\mu \in \M^{+}(\Omega)$.
\item
$f_{+} := \max \{ f, 0 \}$ and $f_{-} := \max\{ -f, 0 \}$.
\end{itemize}
For simplicity, we write $L^{p}(\Omega; dx)$ as $L^{p}(\Omega)$, where $dx$ is the Lebesgue measure.
For measures $\mu$ and $\nu$, we denote $\nu \le \mu$ if $\mu - \nu$ is a nonnegative measure.
For a sequence of extended real valued functions $\{ f_{j} \}_{j = 1}^{\infty}$,
we denote $f_{j} \uparrow f$
if $f_{j + 1} \ge f_{j}$ for all $j \ge 1$ and $\lim_{j \to \infty} f_{j} = f$.
The letter $C$ denotes various constants with and without indices.

\section{Sobolev spaces, capacity, smooth measures}\label{sec:PT}

We first recall some properties of Sobolev spaces from \cite{MR0350027,MR2777530,MR2305115}.
The Sobolev space $H^{1}(\Omega)$ is the set of all weakly differentiable functions $u$
such that $\| u \|_{H^{1}(\Omega)}$ is finite, where 
\[
\| u \|_{H^{1}(\Omega)}
:=
\left(
\int_{\Omega} |u|^{2} + |\nabla u|^{2} \, dx
\right)^{\frac{1}{2}}.
\]
Below, we assume the boundedness of $\Omega$.
We denote by $H_{0}^{1}(\Omega)$ the closure of $C_{c}^{\infty}(\Omega)$ in $H^{1}(\Omega)$.
By the Poincar\'{e} inequality, we can take $\| \nabla \cdot \|_{L^{2}(\Omega)}$ 
as the norm of $H_{0}^{1}(\Omega)$.
For simplicity, we denote by $(\cdot, \cdot)$ the inner product of $H_{0}^{1}(\Omega)$.
Also, we denote by $\langle \cdot, \cdot \rangle$
the duality pairing between $H^{-1}(\Omega)$ and $H_{0}^{1}(\Omega)$.

Let $O \subset \R^{N}$ be open, and let $K \subset O$ be compact.
We define the \textit{(variational) capacity} $\capacity(K, O)$ 
of the condenser $(K, O)$ by
\[
\capacity(K, O)
:=
\inf \left\{
\| \nabla u \|_{L^{2}(O)}^{2} \colon u \geq 1 \ \text{on} \ K, \ u \in C_{c}^{\infty}(O)
\right\}.
\]
For $E \subset O$, we define
\[
\capacity(E, O)
:=
\inf_{ \substack{ E \subset U \subset O \\ U: \text{open} } } \sup_{ \substack{ K \subset U \\ K: \text{compact} } }
\capacity(K, O).
\]
For $E \subset \R^{N}$, the \textit{Sobolev capacity} of $E$ is defined by
\[
C(E)
=
\inf
\int_{\R^{n}} (|u|^{2} + |\nabla u|^{2}) \, dx,
\]
where the infimum is taken over all
$u \in H^{1}(\R^{N})$ such that $u = 1$ in a neighborhood of $E$.
If $\Omega$ is bounded and $E \subset \Omega$, $C(E) = 0$ if and only if $\capacity(E, \Omega) = 0$.
We say that a property holds \textit{quasieverywhere} (q.e.)
if it holds except on a set of capacity zero.
An extended real valued function $u$ on $\Omega$ is called as \textit{quasicontinuous}
if for every $\epsilon > 0$ there exists an open set $E$ such that
$C(E) < \epsilon$ and $u|_{\Omega \setminus E}$ is continuous.
Every $u \in H^{1}_{\loc}(\Omega)$
has a quasicontinuous representative $\tilde{u}$
such that $u = \tilde{u}$ a.e. Below, we denote $\tilde{u}$ by $u$ again for simplicity of notation.

We denote by $\M^{+}_{0}(\Omega)$ the set of all Radon measures $\sigma$
that are absolutely continuous with respect to the capacity, that is,
for any Borel set $E$,
\[
\capacity(E, \Omega) = 0 \ \implies \sigma(E) = 0.
\]
For simplicity, we use the following notations:
\begin{align*}
& H^{-1}(\Omega)^{+} := H^{-1}(\Omega) \cap \M^{+}(\Omega),
\\
& \M_{0, b}^{+}(\Omega) := \{ \sigma \in \M_{0}^{+}(\Omega) \colon \sigma(\Omega) < \infty \}.
\end{align*}

Let us introduce a class of smooth measures with compact support.  

\begin{definition}
We define a class of smooth measures $S_{c}(\Omega)$ by
\[
S_{c}(\Omega)
:=
\left\{
\sigma \in \M^{+}(\Omega)
\colon
\sup_{\Omega} G_{\Omega} \sigma < \infty \ \text{and} \ \spt \sigma \Subset \Omega
\right\},
\]
where $G_{\Omega} \sigma(x) := \int_{\Omega} G_{\Omega}(x, y) \, d \sigma(y)$ and
$G_{\Omega}(\cdot, \cdot)$ is the Green function for the Dirichlet Laplace operator.
\end{definition}

Note that the Green function $G_{\Omega}$ exists since $\Omega$ is bounded.
The definition of $S_{c}(\Omega)$ is inspired by the theory of Dirichlet forms (see, \cite{MR2778606}),
but for the sake of later discussion, we assume the compactness of $\spt \sigma$ unlike usual.
If $\sigma \in S_{c}(\Omega)$, then the energy of $\sigma$,
\[
\int_{\Omega} G_{\Omega} \sigma \, d \sigma
\]
is finite.
Therefore,
\[
S_{c}(\Omega)
\subset
H^{-1}(\Omega)^{+}
\subset
\M^{+}_{0}(\Omega).
\]

The following decomposition theorem holds.
See, \cite[Theorem 2.2.4]{MR2778606} or \cite[Theorem 3.6]{MR4309811}.

\begin{theorem}\label{thm:approximation}
Let $\sigma \in \M^{+}(\Omega)$.
Then, $\sigma \in \M^{+}_{0}(\Omega)$ if and only if
there exists an increasing sequence of compact sets $\{ F_{k } \}_{k = 1}^{\infty}$ such that
$\sigma_{k} := \mathbf{1}_{F_{k}} \sigma \in S_{c}(\Omega)$ for all $k \ge 1$
and
$\sigma\left( \Omega \setminus \bigcup_{k = 1}^{\infty} F_{k} \right) = 0$.
\end{theorem}

\begin{lemma}\label{lem:embedding}
Assume that $\sigma \in S_{c}(\Omega)$.
Then, the embedding 
\begin{equation*}\label{eqn:embedding_L2}
H_{0}^{1}(\Omega) \hookrightarrow L^{2}(\Omega; \sigma)
\end{equation*}
is continuous, and the embedding 
\begin{equation}\label{eqn:embedding_L1}
H_{0}^{1}(\Omega) \hookrightarrow L^{1}(\Omega; \sigma)
\end{equation}
is compact.
\end{lemma}

\begin{proof}
Recall the following Picone-type inequality:
\begin{equation}\label{eqn:picone}
\nabla \left( v^{2} u^{-1} \right) \cdot \nabla u \le |\nabla v|^{2},
\quad
\forall u, v \in C^{1}(\Omega), \ v \ge 0, \ u > 0
\end{equation}
(see, e.g., \cite{MR3273896}).
Setting $M := \| G_{\Omega} \sigma \|_{L^{\infty}(\Omega)}$ and using \eqref{eqn:picone}, we get
\begin{equation}\label{eqn:picone}
\begin{split}
\int_{\Omega} \varphi^{2} \, d \sigma
\le
M \int_{\Omega} \varphi^{2} \, \frac{d \sigma}{G_{\Omega} \sigma}
\le
M \int_{\Omega} |\nabla \varphi|^{2} \, dx,
\quad \forall \varphi \in H_{0}^{1}(\Omega).
\end{split}
\end{equation}
Since $\sigma(\Omega)$ is finite, \eqref{eqn:embedding_L1} is also valid.
The compactness of \eqref{eqn:embedding_L1} follows from 
a general result for trace inequalities.
Since $\sigma$ has compact support, we can take $\eta \in C_{c}^{\infty}(\Omega)$ such that
$\eta \equiv 1$ on $\spt \sigma$.
Therefore,
\[
\begin{split}
\| u \|_{L^{1}(\Omega; \sigma)}
& =
\| \eta u \|_{L^{1}(\Omega; \sigma)}
\\
& \le
(\sigma(\Omega) M)^{\frac{1}{2}} \| \nabla (\eta u) \|_{L^{2}(\Omega)}
\\
& \le
C \| u \|_{H^{1}(\R^{N})},
\quad
\forall u \in H^{1}(\R^{N}).
\end{split}
\]
Then, \cite[Theorem 11.9.2]{MR2777530} yields the desired compactness.
\end{proof}

\begin{corollary}\label{cor:a.e.conv}
Let $\sigma \in \M_{0}^{+}(\Omega)$.
Let $\{ u_{j} \} \subset H_{0}^{1}(\Omega)$, and let $u \in H_{0}^{1}(\Omega)$.
Assume that $u_{j} \wkto u$ weakly in $H_{0}^{1}(\Omega)$.
Then, there exists a subsequence of $\{ u_{j} \}$ such that
$u_{j} \to u$ $\sigma$-a.e. in $\Omega$.
\end{corollary}

\begin{proof}
Let $\{ F_{k} \}$ be a sequence of compact sets in Theorem \ref{thm:approximation}.
By the compact embedding result in Lemma \ref{lem:embedding},
we can choose a subsequence of $\{ u_{j} \}$ such that
$u_{j}(x) \to u(x)$ for $\sigma$-a.e. $x \in F_{1}$.
By a diagonalization argument, we can choose a subsequence of $\{ u_{j} \}$ such that
$u_{j}(x) \to u(x)$ for $\sigma$-a.e. $x \in \bigcup_{k = 1}^{\infty} F_{k}$.
This subsequence has the desired property.
\end{proof}

\section{Energy of measures}\label{sec:TR}

In order to control \eqref{eqn:trace_lambda},
we define a distance of measures.

\begin{definition}
Let $0 \le \lambda \le 1$.
We define a subset $\mathcal{E}_{\lambda}(\Omega)$ of $\M_{0}^{+}(\Omega)$ by
\[
\mathcal{E}_{\lambda}(\Omega)
:=
\left\{ \sigma \in \M_{0}^{+}(\Omega) \colon \trinorm{ \sigma }_{\lambda} < \infty \right\},
\]
where
\[
\trinorm{ \sigma }_{\lambda}
:=
\sup
\left\{
\int_{\Omega} |\varphi|^{1 - \lambda} \, d \sigma
\colon
\forall\varphi \in H_{0}^{1}(\Omega),
\
\| \nabla \varphi \|_{L^{2}(\Omega)} \le 1
\right\}
\]
for $0 \le  \lambda < 1$, and
\[
\trinorm{ \sigma }_{1}
:=
\sigma(\Omega).
\]
For $\sigma, \nu \in \mathcal{E}_{\lambda}(\Omega)$ we define
\[
\mathbf{d}_{\lambda}(\sigma, \nu)
:=
\trinorm{ \, |\sigma - \nu| \, }_{\lambda},
\]
where $|\sigma - \nu|$ is the total variation of $\sigma - \nu$.
\end{definition}

Clearly, $\mathcal{E}_{0}(\Omega) = H^{-1}(\Omega)^{+}$ and
$\mathcal{E}_{1}(\Omega) = \M_{0, b}^{+}(\Omega)$.
By H\"{o}lder's inequality,
\begin{equation}\label{eqn:interpolation}
\trinorm{\sigma}_{\lambda}
\le
\trinorm{\sigma}_{1}^{\lambda} \trinorm{\sigma}_{0}^{1 - \lambda}
=
\sigma(\Omega)^{\lambda} \| \sigma \|_{H^{-1}(\Omega)}^{1 - \lambda}.
\end{equation}
In particular,
\[
S_{c}(\Omega)
\subset
\mathcal{E}_{\lambda}(\Omega)
\subset
\M^{+}_{0}(\Omega).
\]
By definition and the inequality $|\sigma - \nu| \le \sigma + \nu$,
\[
\mathbf{d}_{\lambda}(\sigma, \nu)
\le
\trinorm{ \sigma + \nu }_{\lambda}
\le
\trinorm{ \sigma }_{\lambda} + \trinorm{ \nu }_{\lambda},
\quad \forall \sigma, \nu \in \mathcal{E}_{\lambda}.
\]
Similarly, since $|\sigma - \nu| \le |\sigma - \mu| + |\mu - \nu|$, we have
\[
\mathbf{d}_{\lambda}(\sigma, \nu) 
\le
\mathbf{d}_{\lambda}(\sigma, \mu) + \mathbf{d}_{\lambda}(\mu, \nu),
\quad
\forall \sigma, \nu, \mu \in \mathcal{E}_{\lambda}.
\]
Meanwhile,
\[
\mathbf{d}_{\lambda}(\sigma, \mu) = 0 \implies |\sigma - \nu| = 0 \implies \sigma = \nu.
\]

There are already studies that characterize
the $L^{2}(dx)$-$L^{1 - \lambda}(d \sigma)$ trace inequality \eqref{eqn:trace_lambda}.
Maz'ya and Netrusov \cite{MR1313906} (see also \cite[Theorem 11.6.1]{MR2777530}) gave
a capacitary  characterization for \eqref{eqn:trace_lambda}.
Verbitsky \cite{MR1747901}
gave another characterization which is based on an inequality of the type \eqref{eqn:picone}
(see, also \cite{MR3706136,MR3881877} for related results).
More precisely, as in \cite[Theorem 1.3]{MR4309811}, the following two-sided estimate holds:
\begin{equation}\label{eqn:COV}
\trinorm{ \sigma }_{\lambda}
\le
\left(
\int_{\Omega} (G_{\Omega} \sigma )^{ \frac{1 - \lambda}{1 + \lambda}} \, d \sigma
\right)^{\frac{1 + \lambda}{2}}
\le
\frac{1}{ (1 - \lambda^{2})^{\frac{1 - \lambda}{2}} }
\trinorm{ \sigma }_{\lambda}.
\end{equation}
This inequality is valid even if $\lambda = 1$ clearly.

\begin{proposition}\label{prop:conv}
Let $\sigma \in \mathcal{E}_{\lambda}(\Omega)$.
\begin{enumerate*}[label=(\roman*)]
\item
If $\{ f_{k} \} \subset L^{\infty}(\Omega; \sigma)$ is a bounded sequence of nonnegative functions,
and if $f_{k} \to f$ $\sigma$-a.e. in $\Omega$,
then, $\mathbf{d}_{\lambda}(f \sigma, f_{k} \sigma) \to 0$.
\item
There exists a sequence $\{ \sigma_{k} \} \subset S_{c}(\Omega)$ such that
$\sigma_{k} \uparrow \sigma$ and $\mathbf{d}_{\lambda}(\sigma, \sigma_{k}) \to 0$.
\end{enumerate*}
\end{proposition}

\begin{proof}
Set $g_{k} = |f - f_{k}|$. By assumption, there is a constant $M$ independent of $k$ such that
$g_{k}(x) \le M$ for $\sigma$-a.e. $x \in \Omega$. Therefore,
\[
0 \le G_{\Omega}[g_{k} \sigma](x)
\le
M G_{\Omega}\sigma(x),
\quad \forall x \in \Omega.
\]
By \eqref{eqn:COV} and the dominated convergence theorem,
\[
\int_{\Omega} G_{\Omega}[g_{k} \sigma]^{ \frac{1 - \lambda}{1 + \lambda} } g_{k} \, d \sigma
\to
0.
\]
The second statement follows from Theorem \ref{thm:approximation}.
\end{proof}

The following proposition shows that $\mathcal{E}_{\lambda}(\Omega)$ has a rich structure.

\begin{proposition}
Let $0 \le \lambda \le 1$.
\begin{enumerate*}[label=(\roman*)]
\item\label{cond:HtoE}
If $\mu \in H^{-1}(\Omega)^{+}$, then
$(G_{\Omega} \mu)^{\lambda} \mu \in \mathcal{E}_{\lambda}(\Omega)$.
\item\label{cond:EtoH}
If $\sigma \in \mathcal{E}_{\lambda}(\Omega) \setminus \{ 0 \}$,
then $(G_{\Omega} \sigma)^{ \frac{- \lambda}{1 + \lambda} } \sigma \in H^{-1}(\Omega)^{+}$.
\end{enumerate*}
\end{proposition}

\begin{proof}
Since $\sigma \in \M_{0}^{+}(\Omega)$,
the measures $(G_{\Omega} \mu)^{\lambda} \mu$
and $(G_{\Omega} \sigma)^{ \frac{- \lambda}{1 + \lambda} } \sigma$
are well-defined in the sense of integrals of quasicontinuous functions.
\ref{cond:HtoE}
If $\lambda = 1$, then the statement holds clearly.
Assume that $0 \le \lambda < 1$. By H\"{o}lder's inequality,
\[
\begin{split}
\int_{\Omega} |\varphi|^{1 - \lambda} (G_{\Omega} \mu)^{\lambda} \, d \mu
& \le
\left( \int_{\Omega} G_{\Omega} \mu \, d \mu \right)^{\lambda}
\left( \int_{\Omega} |\varphi|  \, d \mu \right)^{1 - \lambda}
\\
& \le
\| \sigma \|_{H^{-1}(\Omega)}^{1 + \lambda} \| \nabla \varphi \|_{L^{2}(\Omega)}^{1 - \lambda},
\quad \forall \varphi \in H_{0}^{1}(\Omega).
\end{split}
\]
\ref{cond:EtoH}
Fix $\varphi \in C_{c}^{\infty}(\Omega)$ such that $\varphi \ge 0$ in $\Omega$.
Let $\{ \sigma_{k} \} \subset S_{c}(\Omega)$ be a sequence of measures in Theorem \ref{thm:approximation}.
By \eqref{eqn:picone} and \eqref{eqn:COV},
\[
\begin{split}
\int_{\Omega} \frac{ \varphi }{ (G_{\Omega} \sigma_{k})^{ \frac{\lambda}{1 + \lambda} } } \, d \sigma_{k}
& \le
\left(
\int_{\Omega} (G_{\Omega} \sigma_{k} )^{ \frac{1 - \lambda}{1 + \lambda}} \, d \sigma_{k}
\right)^{\frac{1}{2}}
\left(
\int_{\Omega} \varphi^{2} \, \frac{d \sigma_{k}}{G_{\Omega} \sigma_{k}}
\right)^{\frac{1}{2}}
\\
& \le
\frac{1}{ (1 - \lambda^{2})^{\frac{1 - \lambda}{2 (1 + \lambda)}} }
\trinorm{ \sigma }_{\lambda}^{\frac{1}{1 + \lambda}}
\left(
\int_{\Omega} |\nabla \varphi|^{2} \, dx
\right)^{\frac{1}{2}}.
\end{split}
\]
The monotone convergence theorem yields,
\[
\begin{split}
\int_{\Omega} \frac{ \varphi }{ (G_{\Omega} \sigma)^{ \frac{\lambda}{1 + \lambda} } } \, d \sigma
& \le
\frac{1}{ (1 - \lambda^{2})^{\frac{1 - \lambda}{2 (1 + \lambda)}} }
\trinorm{ \sigma }_{\lambda}^{\frac{1}{1 + \lambda}}
\left(
\int_{\Omega} |\nabla \varphi|^{2} \, dx
\right)^{\frac{1}{2}}.
\end{split}
\]
Then, the assertion follows
from the inequality $| \nabla |\cdot| | \le |\nabla \cdot|$.
\end{proof}

Finally, we observe examples of functions in $\mathcal{E}_{\lambda}(\Omega)$.
For one-dimensional cases, we refer to the paper by Sinnamon and Stepanov \cite{MR1395069}.

\begin{example}\label{example:BO}
Suppose that $\Omega \subset \R^{N}$ is bounded. Set
\begin{equation}\label{eqn:sobolev}
2^{*}
:=
\begin{cases}
\frac{2N}{N - 2} & N \ge 3,
\\
\text{any finite constant} & N = 2,
\\
+ \infty & N = 1.
\end{cases}
\end{equation}
Let $0 \le f \in L^{m}(\Omega)$, where $m$ is a constant satisfying \eqref{eqn:cond_BO}.
By Sobolev's inequality and H\"{o}lder's inequality,
\[
\begin{split}
\int_{\Omega} |\varphi|^{1 - \lambda} f \, dx
& \le
C \| f \|_{L^{m}(\Omega)} \| \varphi \|_{L^{2^{*}}(\Omega)}^{1 - \lambda}
\\
& \le
C' \| f \|_{L^{m}(\Omega)} \| \nabla \varphi \|_{L^{2}(\Omega)}^{1 - \lambda},
\quad \forall \varphi \in H_{0}^{1}(\Omega).
\end{split}
\]
Therefore, $\trinorm{f}_{\lambda} \le C \| f \|_{L^{m}(\Omega)}$.
In particular, if $f_{k} \to f$ strongly in $L^{m}(\Omega)$,
then $\mathbf{d}_{\lambda}(f, f_{k}) \to 0$.
\end{example}

\begin{example}
Assume that the following Hardy inequality holds:
\begin{equation}\label{eqn:hardy}
\int_{\Omega} \frac{ \varphi^{2} }{\delta^{2}} \, dx \le C \int_{\Omega} |\nabla \varphi|^{2} \, dx,
\quad
\forall \varphi \in C_{c}^{\infty}(\Omega),
\end{equation}
where $\delta(x) := \dist(x, \del \Omega)$.
Then, for any $0 \le \lambda \le 1$,
\[
\begin{split}
\int_{\Omega} |\varphi|^{1 - \lambda} f \, dx
& \le
\| f \delta^{1 - \lambda} \|_{L^{ \frac{2}{1 + \lambda} }(\Omega)}
\left\| |\varphi|^{1 - \lambda} \delta^{\lambda - 1} \right\|_{L^{\frac{2}{1 - \lambda}}(\Omega)}
\\
& =
\| f \|_{L^{\frac{2}{1 + \lambda}}(\Omega; w_{\lambda})}
\left\| \frac{\varphi}{\delta} \right\|_{L^{2}(\Omega)}^{1 - \lambda}
\\
& \le
C^{1 - \lambda} \| f \|_{L^{\frac{2}{1 + \lambda}}(\Omega; w_{\lambda})}
\| \nabla \varphi \|_{L^{2}(\Omega)}^{1 - \lambda},
\quad \forall \varphi \in H_{0}^{1}(\Omega),
\end{split}
\]
where $w_{\lambda} := \delta^{ \frac{2(1 - \lambda)}{1 + \lambda}}$.
Thus,
$\trinorm{f}_{\lambda} \le C \| f \|_{L^{ \frac{2}{1 + \lambda} }(\Omega; w_{\lambda})}$.
We note that there is no inclusion between
$L^{ \frac{2}{1 + \lambda} }(\Omega; w_{\lambda})$ and $L^{1}(\Omega)$
(see, Example \ref{example:ext+int} below).

Lewis \cite{MR946438} and Wannebo  \cite{MR1010807} proved \eqref{eqn:hardy} 
under the capacity density condition
\begin{equation}\label{eqn:CDC}
\exists c > 0 \ \text{s.t.} \
\forall x \in \partial \Omega, \ \forall r > 0, \
\frac{ \capacity(\cl{B(x, r)} \setminus \Omega, B(x, 2r)) }{ \capacity(\cl{B(x, r)}, B(x, 2r)) } \ge c.
\end{equation}
For further references, we refer to  \cite[Chapter 15]{MR2777530}.
\end{example}

\begin{example}\label{example:ext+int}
Assume \eqref{eqn:hardy} and that there exists $\theta \in (0, 1)$ such that
\begin{equation}\label{eqn:int_dist}
\int_{\Omega} \delta^{ - \theta} \,dx < \infty.
\end{equation}
Then, $\delta^{-s} \in L^{ \frac{2}{1 + \lambda} }(\Omega; w_{\lambda})$
(and thus $\delta^{-s} \in \mathcal{E}_{\lambda}(\Omega)$) if $0 \le \lambda \le 1$ and
\[
\lambda
\le
\frac{2 - 2s + \theta}{2 - \theta}.
\]

Two sufficient conditions for \eqref{eqn:int_dist} can be found in \cite{MR1668136}.
If $\Omega$ has interior cone property, then \eqref{eqn:int_dist} holds if and only if $\theta < 1$.
Therefore, $\delta^{-s} \in \mathcal{E}_{\lambda}(\Omega)$ if \eqref{eqn:DHR} holds.
In particular, $\delta^{-1}$ is not integrable, but $\delta^{-1} \in \mathcal{E}_{\lambda}(\Omega)$
for any $0 \le \lambda < 1$.
The other condition is more complicated.
A bounded domain $\Omega$ is called \textit{John domain}
if there are $x_{0} \in \Omega$ and $C_{J} \ge 1$ such that
each $x \in \Omega$ can be connected to $x_{0}$ by a rectifiable curve $\gamma \subset \Omega$
with
\[
l( \gamma(x, z) ) \le C_{J} \, \delta(z) \quad \forall z \in \gamma,
\]
where $\gamma(x, z)$ is the subarc of $\gamma$ from $x$ to $z$
and $l(\gamma(x, z))$ is the length of $\gamma(x, z)$.
Trocenko \cite{MR624419} proved that if $\Omega$ is a bounded John domain,
then \eqref{eqn:int_dist} holds for some $\theta > 0$.
For another more geometric proof, see \cite[Lemma 4.4]{MR2125540}.
In this case, $\delta^{-\theta} \in \mathcal{E}_{\lambda}(\Omega)$ for all $0 \le \lambda \le 1$.

We note that these conditions are independent of \eqref{eqn:CDC},
because \eqref{eqn:CDC} is a condition for the outside of $\Omega$.
\end{example}

\section{Elliptic equations with singular nonlinearity}\label{sec:SEP}

Fix constants $0 < \alpha \le \beta < \infty$.
We define the set of $N \times N$-matrix-valued functions
$M(\alpha, \beta, \Omega)$ by
\[
M(\alpha, \beta, \Omega)
:=
\left\{
A \in L^{\infty}(\Omega)^{N \times N} \colon
\begin{aligned}
A(x) \xi \cdot \xi & \geq \alpha |\xi|^{2},
\\
A^{-1}(x) \xi \cdot \xi & \ge \beta^{-1} |\xi|^{2},
\end{aligned}
\
\begin{aligned}
\forall \xi \in \R^{N}, \
\text{a.e.} \ x \in \Omega
\end{aligned}
\right\}.
\]
If $A \in M(\alpha, \beta, \Omega)$, then
the differential operator
\[
\mathcal{L} u := - \divergence (A(x) \nabla u)
\]
satisfies
\begin{align}
\left| \langle \mathcal{L} u, v \rangle \right|
& \le
\beta \| \nabla u \|_{L^{2}(\Omega)} \| \nabla v \|_{L^{2}(\Omega)},
\quad
\forall u, v \in H_{0}^{1}(\Omega),
\label{eqn:boundedness}
\\
\langle \mathcal{L} u, u \rangle 
&\ge
\alpha \| \nabla u \|_{L^{2}(\Omega)}^{2},
\quad
\forall u \in H_{0}^{1}(\Omega).
\label{eqn:coercivity}
\end{align}
By the Lax-Milgram theorem,
for any $f \in H^{-1}(\Omega)$, there exists a unique function $u \in H_{0}^{1}(\Omega)$ satisfying
\[
\langle \mathcal{L} u, \varphi \rangle = \langle f, \varphi \rangle
\quad \forall \varphi \in H_{0}^{1}(\Omega).
\]
Assume also that $f$ is a nonnegative distribution.
Then, $u$ is a supersolution to $\mathcal{L}u = 0$ in $\Omega$,
and its \textit{lsc-regularization}
\[
u^{*}(x)
:=
\lim_{r \to 0} \essinf_{B(x, r)} u
\]
gives a quasicontinuous representative of $u$.
For details and more information, including De Giorgi-Nash-Moser theory, we refer to \cite{MR2305115}.
We always use this representative below.

The pointwise estimate for Green functions of divergence form elliptic operators
was established in \cite{MR161019} (see also \cite{MR657523} for non-symmetric $A$).
The following two-sided estimate is one corollary of it.

\begin{theorem}[{\cite[Corollary 7.1]{MR161019}}]\label{thm:LSW}
Let $A \in M(\alpha, \beta, \Omega)$, and let $K$ be a compact subset in $\Omega$.
Assume that $\sigma$ is a nonnegative Radon measure in $H^{-1}(\Omega)$
such that $\spt \sigma \subset K$. 
Let $U$ be the weak solution to
\begin{equation}\label{eqn:poisson}
\begin{cases}
\displaystyle
- \divergence (A(x) \nabla U) = \sigma \quad \text{in} \ \Omega,
\\
U \in H_{0}^{1}(\Omega).
\end{cases}
\end{equation}
Then
\[
\frac{1}{C} G_{\Omega} \sigma
\le
U
\le C G_{\Omega} \sigma,
\quad \text{on} \ K,
\]
where $C$ is a positive constant depending only on $N$, $\alpha$, $\beta$, $\Omega$ and $K$.
\end{theorem}

\begin{remark}\label{rem:boundary_behavior}
Note that the constant $C$ is depending on $K$.
The boundary behavior of $U$ is not comparable to $G_{\Omega} \sigma$ in general
(see, e.g., \cite[Problem 3.9]{MR1814364}).
However, if $\sigma \in S_{c}$, then $U$ is bounded on $\Omega$ from 
Theorem \ref{thm:LSW} and the complete maximum principle.
\end{remark}

We understand Eq. \eqref{eqn:DP} by the following weak sense.

\begin{definition}
Let $A \in M(\alpha, \beta, \Omega)$, $0 < \lambda \le 1$ and $\sigma \in \M_{0}^{+}(\Omega) \setminus \{ 0 \}$.
We say that a function $u \in H_{0}^{1}(\Omega)$ is a (finite energy) weak solution
to \eqref{eqn:DP} if $u > 0$ in $\Omega$ and
\[
\int_{\Omega} A \nabla u \cdot \nabla \varphi \, dx = \int_{\Omega} \frac{\varphi}{u^{\lambda}} \, d \sigma
\quad \forall \varphi \in C_{c}^{\infty}(\Omega).
\]
\end{definition}

\begin{remark}
From the structure of the lower order term,
we may assume that $\sigma$ is absolutely continuous with respect to the capacity without loss of generality.
For details, see \cite[Theorem 3.4]{MR2592976} and \cite[Remark 2.2]{MR3712944}.
\end{remark}

If there exists a weak solution $u$ to \eqref{eqn:DP} , then the embedding
$H_{0}^{1}(\Omega) \hookrightarrow L^{1}(\Omega; \sigma / u^{\lambda})$ is continuous.
Thus, it follows from the completeness of $L^{1}(\Omega; \sigma / u^{\lambda})$
that every $\varphi \in H_{0}^{1}(\Omega)$ is integrable with respect to $\sigma / u^{\lambda}$.
If $v$ is a weak solution to 
\begin{equation}\label{eqn:DPG}
\begin{cases}
\displaystyle
- \divergence (A(x) \nabla v) = \frac{ \nu }{v^{\lambda}}, \quad v > 0 \quad \text{in} \ \Omega,
\\
v \in H_{0}^{1}(\Omega),
\end{cases}
\end{equation}
and if $\sigma \le \nu$, then
\[
\int_{\Omega} A \nabla (u - v) \cdot \nabla (u - v)_{+} \, dx
=
\int_{\Omega} \left( \frac{f}{u^{\lambda}} - \frac{1}{v^{\lambda}} \right) (u - v)_{+} \, d \nu \le 0,
\]
where $f = \frac{d \sigma}{d \nu}$ is the Radon-Nikod\'{y}m derivative of $\sigma$ with respect to $\nu$.
By \eqref{eqn:coercivity}, $\nabla (u - v)_{+} = 0$ a.e. in $\Omega$.
Consequently, $u \le v$ in $H_{0}^{1}(\Omega)$.
In particular, if there exists a weak solution to \eqref{eqn:DP}, it is unique in $H_{0}^{1}(\Omega)$.

Below, we always assume that $\sigma \in \mathcal{E}_{\lambda}(\Omega)$.
This assumption is minimal in the following sense.

\begin{theorem}\label{thm:exist}
Let $A \in M(\alpha, \beta, \Omega)$,
and let $0 < \lambda \le 1$.
Then, there exists a unique weak solution $u \in H_{0}^{1}(\Omega)$ to \eqref{eqn:DP}
if and only if $\sigma \in \mathcal{E}_{\lambda}(\Omega) \setminus \{ 0 \}$.
Moreover,
\begin{equation}\label{eqn:energy_bound}
\frac{1}{\beta^{\frac{1}{1 + \lambda}}}
\trinorm{ \sigma }_{\lambda}^{\frac{1}{1 + \lambda}}
\le
\| \nabla u \|_{L^{2}(\Omega)}
\le
\frac{1}{\alpha^{\frac{1}{1 + \lambda}}}
\trinorm{ \sigma }_{\lambda}^{\frac{1}{1 + \lambda}}.
\end{equation}
Assume also that $\sigma \in S_{c}$.
Then, $u$ is bounded in $\Omega$ and
\begin{equation}\label{eqn:pointwise}
u(x) \le (1 + \lambda)^{\frac{1}{1 + \lambda}} U(x)^{\frac{1}{1 + \lambda}},
\quad \forall x \in \Omega,
\end{equation}
where $U$ is the weak solution to \eqref{eqn:poisson}.
\end{theorem}

\begin{proof}
See the proof of \cite[Theorem 1.2 and Corollary 6.5]{hara2021trace}.
\end{proof}

The following estimate play an essential role in the next section.

\begin{theorem}\label{thm:stability}
Let $\sigma, \nu \in \mathcal{E}_{\lambda}(\Omega) \setminus \{ 0 \}$.
Let $u$ be the weak solution to \eqref{eqn:DP}, and let $v$ be the weak solution to \eqref{eqn:DPG}.
Then,
\[
\| \nabla (u - v) \|_{L^{2}(\Omega)}
\le
\frac{1}{\alpha^{\frac{1}{1 + \lambda}}}
\mathbf{d}_{\lambda}(\sigma, \nu)^{\frac{1}{1 + \lambda}}.
\]
\end{theorem}

\begin{proof}
Let $0 < \lambda < 1$.
By assumption,
\[
\int_{\Omega} A \nabla u \cdot \nabla (u - v) \, dx
=
\int_{\Omega} (u - v)_{+} \frac{1}{u^{\lambda}} \, d \sigma
-
\int_{\Omega} (u - v)_{-} \frac{1}{u^{\lambda}} \, d \sigma
\]
and
\[
\int_{\Omega} A \nabla v \cdot \nabla (u - v) \, dx
=
\int_{\Omega} (u - v)_{+} \frac{1}{v^{\lambda}} \, d \nu
-
\int_{\Omega} (u - v)_{-} \frac{1}{v^{\lambda}} \, d \nu.
\]
If $u > v$, then $\frac{1}{u^{\lambda}} <  \frac{1}{v^{\lambda}}$. Thus,
\[
\begin{split}
\int_{\Omega} (u - v)_{+} \frac{1}{u^{\lambda}} \, d \sigma
-
\int_{\Omega} (u - v)_{+} \frac{1}{v^{\lambda}} \, d \nu
& \le
\int_{\Omega} (u - v)_{+} \frac{1}{u^{\lambda}} \, d (\sigma - \nu)
\\
& \le
\int_{\Omega} (u - v)_{+}^{1 - \lambda} \, d (\sigma - \nu)_{+},
\end{split}
\]
where $(\sigma - \nu)_{+}$ is the positive part of $\sigma - \nu$.
Similarly,
\[
\begin{split}
&
\int_{\Omega} (u - v)_{-} \frac{1}{v^{\lambda}} \, d \nu
-
\int_{\Omega} (u - v)_{-} \frac{1}{u^{\lambda}} \, d \sigma
\le
\int_{\Omega} (u - v)_{-}^{1 - \lambda} \, d (\nu - \sigma)_{+}.
\end{split}
\]
Combining the four inequalities and using \eqref{eqn:coercivity}, we obtain,
\[
\begin{split}
\alpha \int_{\Omega} |\nabla (u - v)|^{2} \, dx
& \le
\int_{\Omega} |u - v|^{1 - \lambda} \, d |\sigma - \nu|
\\
& \le
\trinorm{ \, |\sigma - \nu| \, }_{\lambda} \| \nabla |u - v| \|_{L^{2}(\Omega)}^{1 - \lambda}.
\end{split}
\]
By the inequality $| \nabla |\cdot| | \le |\nabla \cdot|$,
we arrive at the desired estimate.
The same argument is valid even if $\lambda = 1$.
\end{proof}

\section{Homogenization}\label{sec:HG}

Let us recall the definition of $H$-convergence.

\begin{definition}
Let $\{ A_{\epsilon} \} \subset M(\alpha, \beta, \Omega)$,
and let $A_{0} \in L^{\infty}(\Omega)^{N \times N}$.
We say that $A_{\epsilon}$ \textit{H-converges} to $A_{0}$ ($A_{\epsilon} \xrightarrow{H} A_{0}$)
if for every $f \in H^{-1}(\Omega)$,
the sequence $\{ u_{\epsilon} \}$ of weak solutions to
\[
\begin{cases}
- \divergence (A_{\epsilon}(x) \nabla u_{\epsilon}) = f \quad \text{in} \ \Omega,
\\
u_{\epsilon} \in H_{0}^{1}(\Omega),
\end{cases}
\]
satisfies
\begin{align*}
u_{\epsilon} & \wkto u_{0} \ \text{weakly in} \ H_{0}^{1}(\Omega),
\\
A_{\epsilon} \nabla u_{\epsilon} & \wkto A_{0} \nabla u_{0} \  \text{weakly in} \ L^{2}(\Omega)^{N},
\end{align*}
where $u_{0}$ is the weak solution to
\begin{equation}\label{eqn:limit-equ}
\begin{cases}
- \divergence (A_{0}(x) \nabla u_{0}) = f \quad \text{in} \ \Omega,
\\
u_{0} \in H_{0}^{1}(\Omega).
\end{cases}
\end{equation}
\end{definition}

The concept of H-convergence was introduced by Murat and Tartar in the late 1970s
(see, e.g., \cite{MR1493039, MR1859696}).
If $A_{\epsilon} \to A_{0}$ a.e. in $\Omega$, then $A_{\epsilon} \xrightarrow{H} A_{0}$.
Many nontrivial examples can be found in periodic homogenization (e.g., \cite{MR3839345}).
If $\{ A_{\epsilon} \} \subset M(\alpha, \beta, \Omega)$ and
$A_{\epsilon} \xrightarrow{H} A_{0}$, then $A_{0}$ belongs to $M(\alpha, \beta, \Omega)$ again.

\begin{remark}\label{rem:strong_conv}
Assume that $A_{\epsilon} \xrightarrow{H} A_{0}$ and $f_{\epsilon} \to f$ strongly in $H^{-1}(\Omega)$.
Let  $\{ u_{\epsilon} \}$ be the sequence of weak solutions to
\[
\begin{cases}
- \divergence (A_{\epsilon}(x) \nabla u_{\epsilon}) = f_{\epsilon} \quad \text{in} \ \Omega,
\\
u_{\epsilon} \in H_{0}^{1}(\Omega).
\end{cases}
\]
Then, $u_{\epsilon} \wkto u_{0}$ weakly in $H_{0}^{1}(\Omega)$,
where $u_{0}$ is the weak solution to \eqref{eqn:limit-equ}.
The strong convergence of $f_{\epsilon}$ can not be replaced by the weak convergence generally.
Further information can be found in \cite[Section 3.1.5]{MR3839345}.
\end{remark}

Our main result is as follows.

\begin{theorem}\label{thm:main2}
Let $\{ A_{\epsilon} \} \subset M(\alpha, \beta, \Omega)$,
$0 < \lambda \le 1$ and $\{ \sigma_{\epsilon} \} \subset \mathcal{E}_{\lambda}(\Omega) \setminus \{ 0 \}$.
Let $\{ u_{\epsilon} \} \subset H_{0}^{1}(\Omega)$ be the sequence of weak solutions to 
\begin{equation}\label{eqn:DPEE}
\begin{cases}
\displaystyle
- \divergence (A_{\epsilon}(x) \nabla u_{\epsilon})
=
\frac{ \sigma_{\epsilon} }{u_{\epsilon}^{\lambda}},
\quad
u_{\epsilon} > 0
\quad \text{in} \ \Omega,
\\
u_{\epsilon} \in H_{0}^{1}(\Omega).
\end{cases}
\end{equation}
Assume that $A_{\epsilon} \xrightarrow{H} A_{0}$
and that there exists $\sigma \in \mathcal{E}_{\lambda}(\Omega) \setminus \{ 0 \}$ such that
$\mathbf{d}_{\lambda}(\sigma, \sigma_{\epsilon}) \to 0$.
Then, $u_{\epsilon} \wkto u_{0}$ weakly in $H_{0}^{1}(\Omega)$,
where $u_{0}$ is the weak solution to
\begin{equation}\label{eqn:DP0}
\begin{cases}
\displaystyle
- \divergence (A_{0}(x) \nabla u_{0}) = \frac{\sigma}{u_{0}^{\lambda}},
\quad 
u_{0} > 0
\quad
\text{in} \ \Omega,
\\
u_{0} \in H_{0}^{1}(\Omega).
\end{cases}
\end{equation}
Moreover,
$\sigma_{\epsilon} / u_{\epsilon}^{\lambda} \to \sigma / u_{0}^{\lambda}$ strongly in $H^{-1}(\Omega)$.
\end{theorem}

H-convergence can be rewritten as $\Gamma$-convergence by defining an appropriate functional,
and stability of $\Gamma$-convergence with respect to continuous perturbations is well-known
(see, \cite{MR1201152, MR3017993}).
However, our lower-order term does not admit this framework.
Hence, following \cite{MR3177473}, we reduce the problem and use elliptic regularity.

\begin{lemma}\label{lem:homogen}
Let $\{ A_{\epsilon} \} \subset M(\alpha, \beta, \Omega)$, $0 < \lambda \le 1$ and
$\sigma \in S_{c}(\Omega) \setminus \{ 0 \}$.
Let $\{ u_{\epsilon} \} \subset H_{0}^{1}(\Omega)$ be the sequence of weak solutions to 
\begin{equation}\label{eqn:DPE}
\begin{cases}
\displaystyle
- \divergence (A_{\epsilon}(x) \nabla u_{\epsilon}) 
=
\frac{ \sigma }{u_{\epsilon}^{\lambda}},
\quad
u_{\epsilon} > 0
\quad \text{in} \ \Omega,
\\
u_{\epsilon} \in H_{0}^{1}(\Omega).
\end{cases}
\end{equation}
Assume also that $A_{\epsilon}  \xrightarrow{H} A_{0}$.
Then, $u_{\epsilon} \wkto u_{0}$ weakly in $H_{0}^{1}(\Omega)$,
where $u_{0}$ is the weak solution to \eqref{eqn:DP0}.
Moreover, $\sigma / u_{\epsilon}^{\lambda} \to \sigma / u_{0}^{\lambda}$ strongly in $H^{-1}(\Omega)$.
\end{lemma}

\begin{proof}
By assumption and Theorem \ref{thm:exist},
$\| \nabla u_{\epsilon} \|_{L^{2}(\Omega)}$ is bounded.
Denote by $\{ u_{\epsilon} \}$ any subsequence of $\{ u_{\epsilon} \}$ again.
By weak compactness and Lemma \ref{lem:embedding},
we can choose a subsequence of $\{ u_{\epsilon} \}$ such that
$u_{\epsilon} \wkto u_{0}$ weakly in $H_{0}^{1}(\Omega)$
and
$u_{\epsilon} \to u_{0}$ $\sigma$-a.e. in $\Omega$.
Since $\sigma \in S_{c}(\Omega)$, by Theorems \ref{thm:LSW} and \ref{thm:exist},
\[
\| u_{\epsilon} \|_{L^{\infty}(\Omega; \sigma)}
\le
C \| G_{\Omega} \sigma \|_{L^{\infty}(\Omega; \sigma)}^{\frac{1}{1 + \lambda}} (=: M).
\]
Note that the right-hand side is independent of $\epsilon$.
Since
\[
- \divergence (A_{\epsilon}(x) \nabla u_{\epsilon}) 
\ge
\frac{ \sigma }{M^{\lambda}} > 0
\quad \text{in} \ \Omega,
\]
using Theorem \ref{thm:LSW} again, we obtain
\[
u_{\epsilon}(x) \ge \frac{1}{C} G_{\Omega} \sigma(x) \ge m, \quad  \forall x \in \spt \sigma,
\]
where $m$ is a positive constant independent of $\epsilon$.
By the dominated convergence theorem,
\[
\frac{1}{u_{\epsilon}^{\lambda}} \to \frac{1}{u_{0}^{\lambda}} \quad \text{strongly in} \ L^{2}(\Omega; \sigma).
\]
Meanwhile, by Lemma \ref{lem:embedding} and H\"{o}lder's inequality,
\[
\left|
\int_{\Omega}
\varphi \left( \frac{1}{u_{\epsilon}^{\lambda}}  - \frac{1}{u_{0}^{\lambda}} \right) \,
d \sigma
\right|
\le
C
\left\| \frac{1}{u_{\epsilon}^{\lambda}}  - \frac{1}{u_{0}^{\lambda}} \right\|_{L^{2}(\Omega; \sigma)}
\| \nabla \varphi \|_{L^{2}(\Omega)},
\quad
\forall \varphi \in H_{0}^{1}(\Omega).
\]
Thus, $\sigma / u_{\epsilon}^{\lambda} \to \sigma / u_{0}^{\lambda}$ strongly in $H^{-1}(\Omega)$.
Since $A_{\epsilon} \xrightarrow{H} A_{0}$,
it follows from Remark \ref{rem:strong_conv} that $u_{0}$ satisfies \eqref{eqn:DP0}.
By the uniqueness of weak solutions,
the entire sequence $\{ u_{\epsilon} \}$ converges to $u_{0}$ weakly in $H_{0}^{1}(\Omega)$.
\end{proof}

\begin{lemma}\label{lem:homogen2}
The results in Lemma \ref{lem:homogen}
are still valid for $\sigma \in \mathcal{E}_{\lambda}(\Omega) \setminus \{ 0 \}$.
\end{lemma}

\begin{proof}
By Theorem \ref{thm:exist}, $\{ u_{\epsilon} \}$ is bounded in $H_{0}^{1}(\Omega)$.
Moreover, there exists a unique weak solution $u_{0}$ to \eqref{eqn:DP0}.
Using Theorem \ref{thm:approximation},
we take a sequence $\{ \sigma_{k} \} \subset S_{c}(\Omega) \setminus \{ 0 \}$
such that $\sigma_{k} \uparrow \sigma$.
For $\epsilon > 0$ and $\epsilon = 0$, we consider the following approximating problems:
\begin{equation}\label{eqn:DPEK}
\begin{cases}
\displaystyle
- \divergence (A_{\epsilon}(x) \nabla u_{\epsilon}^{k})
=
\frac{ \sigma_{k} }{(u_{\epsilon}^{k})^{\lambda}},
\quad
u_{\epsilon}^{k} > 0
\quad \text{in} \ \Omega,
\\
u_{\epsilon}^{k} \in H_{0}^{1}(\Omega).
\end{cases}
\end{equation}
Fix $\varphi \in H_{0}^{1}(\Omega)$.
By Theorem \ref{thm:stability},
\[
\begin{split}
|(u_{\epsilon} - u_{0}, \varphi )|
& =
|(u_{\epsilon} - u_{\epsilon}^{k}, \varphi )| + |(u_{\epsilon}^{k} - u_{0}^{k}, \varphi )| + |(u_{0}^{k} - u_{0}, \varphi )|
\\
& \le
\frac{2}{\alpha^{\frac{1}{1 + \lambda}}}
\mathbf{d}_{\lambda}(\sigma, \sigma_{k})^{\frac{1}{1 + \lambda}} \| \nabla \varphi \|_{L^{2}(\Omega)}
+
|(u_{\epsilon}^{k} - u_{0}^{k}, \varphi )|.
\end{split}
\]
By Lemma \ref{lem:homogen},
the latter term goes to zero as $\epsilon \to 0$.
Thus,
\[
\limsup_{\epsilon \to 0} |(u_{\epsilon} - u_{0}, \varphi)|
\le
\frac{2}{\alpha^{\frac{1}{1 + \lambda}}}
\mathbf{d}_{\lambda}(\sigma, \sigma_{k})^{\frac{1}{1 + \lambda}} \| \nabla \varphi \|_{L^{2}(\Omega)}.
\]
By Proposition \ref{prop:conv},
the right-hand side goes to zero as $k \to \infty$.
Thus, $u_{\epsilon} \wkto u_{0}$ weakly in $H_{0}^{1}(\Omega)$.
Similarly, by Theorem \ref{thm:stability} and \eqref{eqn:boundedness},
\[
\begin{split}
\| \mathcal{L}_{\epsilon} u_{\epsilon} - \mathcal{L}_{0} u_{0} \|_{H^{-1}(\Omega)}
& \le
\| \mathcal{L}_{\epsilon} u_{\epsilon} - \mathcal{L}_{\epsilon} u_{\epsilon}^{k} \|_{H^{-1}(\Omega)}
+
\| \mathcal{L}_{\epsilon} u_{\epsilon}^{k} - \mathcal{L}_{0} u_{0}^{k} \|_{H^{-1}(\Omega)}
\\
& \quad +
\| \mathcal{L}_{0} u_{0}^{k} - \mathcal{L}_{0} u_{0} \|_{H^{-1}(\Omega)}
\\
& \le
\frac{2 \beta}{\alpha^{\frac{1}{1 + \lambda}}}
\mathbf{d}_{\lambda}(\sigma, \sigma_{k})^{\frac{1}{1 + \lambda}}
+
\| \mathcal{L}_{\epsilon} u_{\epsilon}^{k} - \mathcal{L}_{0} u_{0}^{k} \|_{H^{-1}(\Omega)}.
\end{split}
\]
By Lemma \ref{lem:homogen},
\[
\| \mathcal{L}_{\epsilon} u_{\epsilon}^{k} - \mathcal{L}_{0} u_{0}^{k} \|_{H^{-1}(\Omega)}
=
\left\|
\frac{\sigma_{k}}{(u_{\epsilon}^{k})^{\lambda}} - \frac{\sigma_{k}}{(u_{0}^{k})^{\lambda}} 
\right\|_{H^{-1}(\Omega)}
\to 0
\quad \text{as} \ \epsilon \to 0.
\]
Thus,
\[
\limsup_{\epsilon \to 0} \| \mathcal{L}_{\epsilon} u_{\epsilon} - \mathcal{L}_{0} u_{0} \|_{H^{-1}(\Omega)}
\le
\frac{2 \beta}{\alpha^{\frac{1}{1 + \lambda}}}
\mathbf{d}_{\lambda}(\sigma, \sigma_{k})^{\frac{1}{1 + \lambda}}.
\]
Taking the limit $k \to \infty$,
we find that $\mathcal{L}_{\epsilon} u_{\epsilon} \to \mathcal{L}_{0} u_{0}$ strongly in $H^{-1}(\Omega)$.
\end{proof}


\begin{proof}[Proof of Theorem \ref{thm:main2}]
Let $\{ v_{\epsilon} \}$ be the sequence of weak solutions to \eqref{eqn:DPE}.
Fix $\varphi \in H_{0}^{1}(\Omega)$.
By Theorem \ref{thm:stability} and Lemma \ref{lem:homogen2},
\begin{equation*}
\begin{split}
|(u_{\epsilon} - u_{0}, \varphi )|
& \le
|(u_{\epsilon} - v_{\epsilon}, \varphi )| + |(v_{\epsilon} - u_{0}, \varphi )|
\\ 
& \le
\frac{1}{\alpha^{\frac{1}{1 + \lambda}}}
\mathbf{d}_{\lambda}(\sigma, \sigma_{\epsilon})^{\frac{1}{1 + \lambda}} \| \nabla \varphi \|_{L^{2}(\Omega)}
+
|(v_{\epsilon} - u_{0}, \varphi )|
\to 0
\quad \text{as} \ \epsilon \to 0.
\end{split}
\end{equation*}
Therefore, $u_{\epsilon} \wkto u_{0}$ weakly in $H_{0}^{1}(\Omega)$.
Similarly,
\[
\begin{split}
\| \mathcal{L}_{\epsilon} u_{\epsilon} - \mathcal{L}_{0} u_{0} \|_{H^{-1}(\Omega)}
& \le
\| \mathcal{L}_{\epsilon} u_{\epsilon} - \mathcal{L}_{\epsilon} v_{\epsilon} \|_{H^{-1}(\Omega)}
+
\| \mathcal{L}_{\epsilon} v_{\epsilon} - \mathcal{L}_{0} u_{0} \|_{H^{-1}(\Omega)}
\\
& \le
\frac{\beta}{\alpha^{\frac{1}{1 + \lambda}}}
\mathbf{d}_{\lambda}(\sigma, \sigma_{\epsilon})^{\frac{1}{1 + \lambda}}
+
\| \mathcal{L}_{\epsilon} v_{\epsilon} - \mathcal{L}_{0} u_{0} \|_{H^{-1}(\Omega)}
\to 0.
\end{split}
\]
This completes the proof.
\end{proof}

\begin{remark}
If $A_{\epsilon} \to A_{0}$ a.e. in $\Omega$,
then Theorem \ref{thm:main2} yields the strong convergence of $u_{\epsilon}$ in $H_{0}^{1}(\Omega)$
(see, e.g., \cite[Theorem 1.2.20]{MR1859696}).
\end{remark}


\bibliographystyle{abbrv} 
\bibliography{reference}


\end{document}